\documentclass[runningheads]{llncs}
\usepackage[T1]{fontenc}
\usepackage{amsmath}
\usepackage{amssymb}
\newtheorem{prop}{Proposition}
\newtheorem{thm}{Theorem}
\newtheorem{dfn}{Definition}
\newtheorem{cor}{Corollary}
\newcommand{\cF}{\mathcal{F}}
\newcommand{\cA}{\mathcal{A}}
\usepackage{graphicx}

\begin{document}

\title{Shape Theory via the Atiyah--Molino Reconstruction and Deformations}

%
%
\author{No\'emie C. Combe\inst{1}\orcidID{0000-0003-4540-7376} \and
Hanna K. Nencka }
\authorrunning{N. Combe et al.}
%
\institute{University of Warsaw, Ul. Banacha 2, 02-097 Warsaw, Poland
\email{n.combe@uw.edu.pl}\\
\url{https://noemie-combe-23.webself.net}}
\maketitle              
\begin{abstract}
Reconstruction problems lie at the very heart of both mathematics and science, posing the enigmatic challenge: \emph{How does one resurrect a hidden structure from the shards of incomplete, fragmented, or distorted data?} In this paper, we introduce a new approach that harnesses the profound insights of the Vaisman Atiyah--Molino framework. Our method renders the reconstruction problem computationally tractable while exhibiting exceptional robustness in the presence of noise. Central to our theory is the Hantjies tensor—a curvature-like invariant that precisely quantifies noise propagation and enables error-bounded reconstructions. This synthesis of differential geometry, integral analysis, and algebraic topology not only resolves long-standing ambiguities in inverse problems but also paves the way for transformative applications across a broad spectrum of scientific disciplines.
\end{abstract}

\section{Introduction}

Reconstruction problems stand among the biggest challenges in mathematics and applied sciences: \emph{How does one rebuild a hidden structure from fragmented, incomplete, or distorted data?} Whether it is deducing the three-dimensional conformation of a protein from blurry two-dimensional microscope images, or inferring the structure of a graph from partially observed subgraphs, these inverse problems lie at the core of diverse disciplines—ranging from cryo-electron microscopy and  tomography to astronomy, radio-astronomy, medicine and graph theory as well as modern machine learning.

\, 

The intrinsic difficulty of reconstruction is underscored by the fact that infinitely many structures may be consistent with the same incomplete data. In practice, traditional methods, such as iterative algorithms, are hampered by three fundamental challenges:
\begin{enumerate}
    \item \textbf{Uniqueness:} Does the available data suffice to uniquely determine the original structure?
    \item \textbf{Noise:} How can one mitigate the deleterious effects of noise and measurement imperfections inherent in real-world data?
    \item \textbf{Complexity:} How does one efficiently navigate the vast, often combinatorially complex, space of potential solutions?
\end{enumerate}

\, 

In this work, we introduce a {\it new framework for reconstruction} that {\it unifies} and {\it extends} previous approaches by leveraging the deep insights of the \emph{Vaisman} \cite{V} and \emph{Atiyah--Molino} \cite{A,M} frameworks. Rather than addressing each inverse problem in isolation, our approach employs a unified geometric language that recasts reconstruction as the extraction of hidden symmetries and invariants within the data.

\, 

On one hand, the {\bf Vaisman framework} interprets reconstruction as a geometric unraveling of symmetry. It resolves the inherent multiplicity of solutions in underdetermined inverse problems by inducing a stratified decomposition of the configuration space into transverse foliations. Each foliation represents an equivalence class of solutions that are indistinguishable under a given projection operator, thereby transforming the ambiguity into a structured hierarchy of invariant submanifolds. The intersection of these orthogonal layers of constraints is isolated, which in turn guarantees unique solutions and robustness against noise. For instance, objects that yield identical two-dimensional images naturally lie on the same leaf of the foliation.

On the other hand, the {\bf Atiyah--Molino framework} reinterprets the reconstruction problem as a fiber bundle phenomenon. By splitting the problem into tangent directions (encoding local deformations, such as the tilting of a protein) and normal directions (capturing global invariants such as centroids and moments), this approach achieves significant computational tractability. The Atiyah--Molino exact sequence decomposes the problem into algebraic equations, thereby obviating the need for brute-force searches. Furthermore, the introduction of the Hantjies tensor---a curvature-like invariant---provides a precise measure of noise propagation, thereby allowing for error-bounded reconstructions.

This synthesis of differential geometry, algebraic topology, and integral analysis not only resolves long-standing ambiguities in inverse problems but also paves the way for transformative applications across science and engineering. In the subsequent sections, we develop the theoretical underpinnings of our approach and demonstrate its efficacy in several paradigmatic reconstruction scenarios.

Applications of our method include:

\begin{itemize}
\item {\bf Medicine}: Reconstructing tumours from sparse MRI slices with error guarantees.
\item {\bf Quantum Computing:} Inferring quantum states from noisy partial measurements.
\item {\bf Artificial Intelligence}: Training generative models to impute missing data through geometric manifold learning and differential topological constraints, as opposed to relying on conventional statistical priors or probabilistic inference frameworks
\end{itemize}

\begin{figure}[h]
    \centering
    \includegraphics[width=0.8\textwidth]{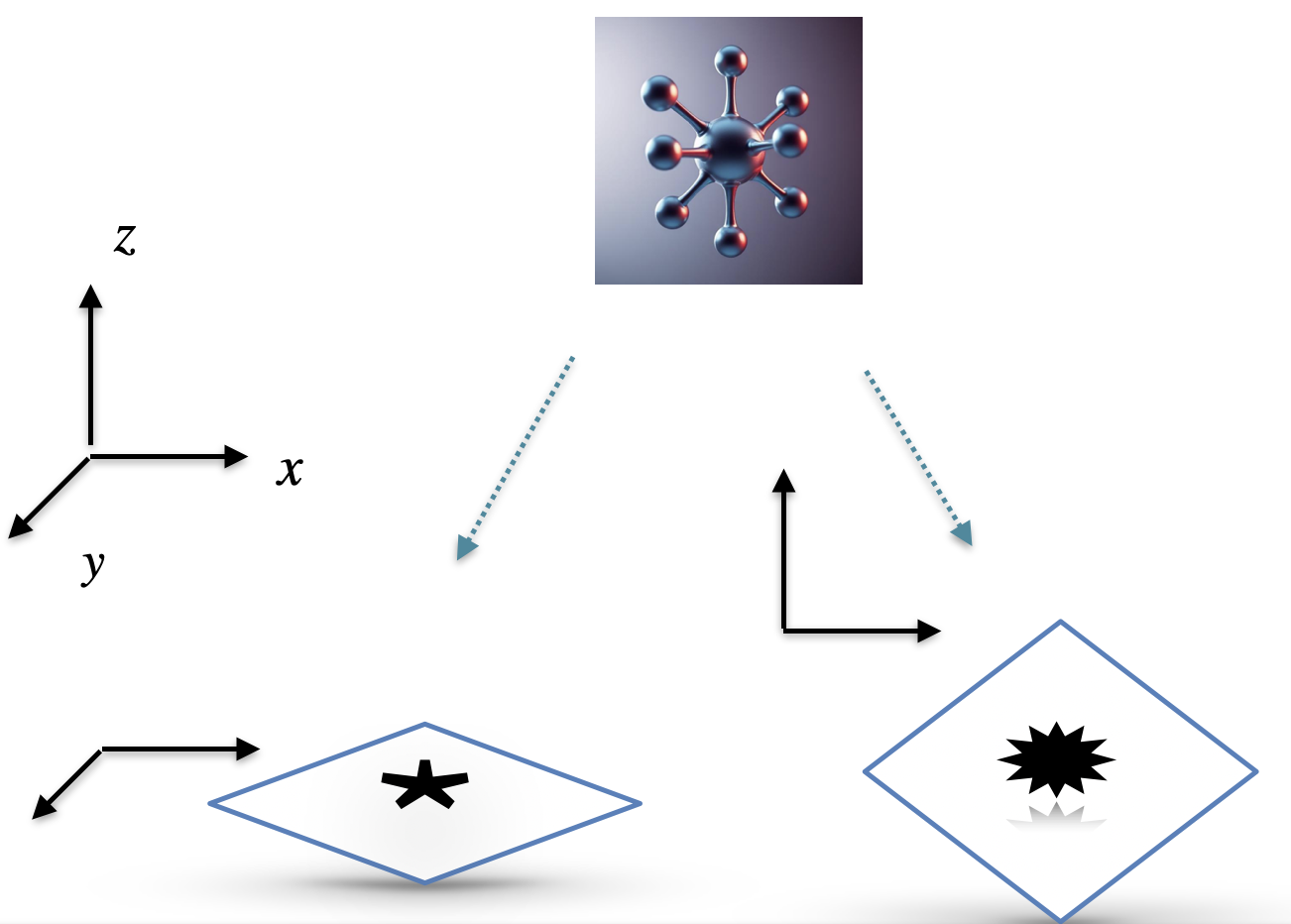}
    \caption{Illustration of an imaging process: each projection image corresponds the line integrals of a molecule rotated by a
three-dimensional rotation.}
    \label{fig:1}
\end{figure}

{\bf Acknowledgements} The first author thanks Pawel Dlotko for discussions on TDA. Both authors thank Philippe Combe for comments on this paper. This research is part of the project No. 2022/47/P/ST1/01177 co-founded by the National Science Centre and the European Union's Horizon 2020 research and innovation program, under the Marie Sklodowska Curie grant agreement No. 945339.For the purpose of Open Access, the author has applied a CC-BY public copyright licence to any Author Accepted Manuscript (AAM) version arising from this submission.
\includegraphics[width=1cm, height=0.5cm]{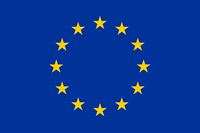}.

\section{Projections and Inverse Morphisms}
\subsection{} We consider at least two independent projections of a three-dimensional object onto distinct two-dimensional planes. The existence of such independent projections is necessary to ensure sufficient information for reconstruction. This leads to the problem of defining an inverse function that maps the planar projections back to the original volumetric object.

\begin{itemize}
\item A single projection collapses three dimensional information onto a plane, losing depth and orientation data. For example, in electron microscopy, a single micrograph of a particle provides no information about its tilt angle relative to the imaging axis.

\item By considering projections along distinct planes, the system gains orthogonal constraints (e.g., tilt angles and rotation axes) that resolve ambiguities. This aligns with Gelfand and Goncharov’s method, which uses statistical properties of projections, such as first moments of plane sections, to derive orientation parameters.
\end{itemize}

\subsection{}  The Radon transform $Rf(\theta,t)$ integrates a function 
$f(x,y,z)$ along planes parameterized by angle 
$\theta$ and offset $t$. Its inverse requires integrating over all possible angles, but practical applications (e.g., cryo-EM) use finite projections. 

 For discrete objects (e.g., particles on a line), the inverse function can be constructed as a linear system where each projection contributes equations. Two independent projections ensure the system is determined (solvable) under non-degenerate conditions \cite{GG}.
\medskip 

\subsection{}  A key idea is to leverage the results of Neifeld on geometric projective 2-dimensional spaces. From these results, it follows that the independent projections induce a pair of independent connections. These connections allow us to extract curvature-related information in two dimensions, including the Riemann tensor and Ricci curvature. Understanding whether these curvature tensors arise from an associative or commutative algebraic structure further informs the reconstruction process.

\, 

Neifeld’s involution principle bridges projective geometry and differential geometry, enabling a dual-connection framework for 3D reconstruction. By interpreting projections as inducing independent connections on $\mathbb{C}P^2$, this approach generalizes classical integral geometry methods and enhances their robustness, particularly in complex or symmetric settings. Further work could explore applications to quantum state tomography (via projective Hilbert spaces) or algebraic varieties in $\mathbb{C}P^n$.

\begin{prop}
~

(i) Let \( O \subset \mathbb{CP}^3 \) be a smooth, non-symmetric three dimensional object embedded in the complex projective space, and let \( \Pi_1, \Pi_2: \mathbb{CP}^3 \dashrightarrow \mathbb{CP}^2 \) be two independent rational projections onto distinct complex projective planes. Assume:

\begin{enumerate}
\item  {\bf Non-degenerate projections}: The restrictions of \( \Pi_1 \) and \( \Pi_2 \) to \( O \) are immersive; i.e., the differentials \( d\Pi_1 \) and \( d\Pi_2 \) have full rank.

\item  {\bf Involution symmetry:} The projections satisfy \[\Pi_2 = \iota \circ \Pi_1 ,\] where \( \iota: \mathbb{CP}^2 \to \mathbb{CP}^2 \) is an involution (e.g., a polarity induced by the Fubini-Study metric).

\end{enumerate}

Each projection \( \Pi_i \) defines a holomorphic line bundle \( L_i \to \mathbb{CP}^2 \) with a connection \( \nabla_i \) derived from the Fubini-Study metric. The involution \( \iota \) induces a duality \[ L_1 \leftrightarrow L_2^* ,\] making \( (\nabla_1, \nabla_2) \) a dual pair.
\end{prop}
\medskip 

\subsubsection{Centroids}

The \emph{centroid} (or first moment) of a geometric shape is its ``average position" in space, computed as the arithmetic mean of all points in the object. For a three dimensional object projected onto a 2D plane, if the object is represented as a set of points or a density distribution, the centroid of its projection onto a plane is the weighted average position of those points in that plane.

For a projection 
\[
\Pi_i: \mathbb{R}^3 \to \mathbb{R}^2,
\]
the centroid \(\bar{p}_i \in \mathbb{R}^2\) is given by
\[
\bar{p}_i = \left( \frac{1}{N}\sum_{k=1}^{N} x_k, \; \frac{1}{N}\sum_{k=1}^{N} y_k \right),
\]
where \((x_k,y_k)\) are the coordinates of the projected points and \(N\) is the total number of points.

The centroid encodes the translational symmetry of the projected data. For example, shifting the three dimensional object in space shifts the centroid linearly in the projection.

\subsubsection{Moment Maps}
A \emph{moment map} generalizes the concept of centroids to algebraic/geometric settings, often encoding symmetry-invariant properties of an object.

The moment map 
\[
\mu_i: O \to \mathbb{C}^2
\]
assigns to the three dimensional object \(O\) the centroid of its projection onto the plane \(\Pi_i\). If \(O\) is a density distribution, \(\mu_i\) computes the first statistical moment (mean) of the projection. For a line or curve, \(\mu_i\) corresponds to the centroid of its projected trace.

Each \(\mu_i\) provides a linear constraint on the orientation of \(O\). Combining \(\mu_1\) and \(\mu_2\) (from two distinct projections) resolves ambiguities in the three dimensional orientation.

\begin{prop}\label{P:2}
Let \( \mu_i: O \to \mathbb{C}^2 \) be the first moment maps of \( O \) with respect to \( \Pi_i \), encoding the centroids of the projected data.

\smallskip 

Under these conditions:

\begin{itemize}
\item  {\bf Uniqueness}: The orientation of \( O \) (i.e., its position modulo projective transformations) is uniquely determined by the compatibility of \( \nabla_1 \) and \( \nabla_2 \) acting on \( \mu_1 \) and \( \mu_2 \).

\item {\bf Reconstruction}: The original object \( O \) can be reconstructed as the intersection of the parallel transports along \( \nabla_1 \) and \( \nabla_2 \), applied to the moment maps \( \mu_1 \) and \( \mu_2 \).
\end{itemize}
Explicitly, there exists a unique solution \( v \in T\mathbb{CP}^3 \) (a direction vector modulo scaling) satisfying:

\begin{equation}\label{E:1}
\begin{cases}
\nabla_1 \mu_1 = v \cdot \omega_1, \\
\nabla_2 \mu_2 = v \cdot \omega_2,
\end{cases}
\end{equation}

where \( \omega_i \) are connection 1-forms (local representatives of the connections) encoding the involution duality: \( \omega_1 = \iota^* \omega_2 \), where $\iota:\mathbb{C}P^3\to \mathbb{C}P^3$ is an involution. 
\end{prop}

The system of equations in Eq.\eqref{E:1}
is a pair of linear constraints on $v$ where $v \cdot \omega_i$ represents the contraction of $v$ with the connection 1-form $\omega_i$, inducing a scalar equation. 

Note that the involution $\iota$ induces a duality between the connections:
\[\omega_1=\iota^*\omega_2\Longrightarrow \nabla_1=\iota^*\nabla_2.\]

The system has a unique solution \(v\) (modulo scaling) for the following reasons: 

\begin{itemize}
    \item \textbf{Transversality:} The projections \(\Pi_1\) and \(\Pi_2\) are independent (non-coaxial), so their constraints on \(v\) intersect transversally.
    
    \item \textbf{Algebraic Rank:} The equations define a full-rank linear system. For example, if 
    \[
    v = (v_0, v_1, v_2, v_3) \in \mathbb{C}^4,
    \]
    the two equations reduce the degrees of freedom from 4 (modulo scaling) to 1.
    
    \item \textbf{Curvature Compatibility:} The duality 
    \[
    \omega_1 = \iota^*\omega_2
    \]
    ensures that the curvatures 
    \[
    F_{\nabla_1} = -F_{\nabla_2}
    \]
    do not introduce conflicting constraints.
\end{itemize}

\subsection{Proof of  Proposition \ref{P:2}}
Before we proceed to the proof of the Prop. \ref{P:2}, we recall the notion of a curvature.  Given a connection $\nabla$ on a vector bundle $E\to M$, its curvature $F_{\nabla}$ is the 2-form-valued endomorphism measuring the failure of $\nabla$ to be flat. Formally:
\[F_{\nabla}(X,Y)=\nabla_X\nabla_Y-\nabla_Y\nabla_X-\nabla_{[X,Y]},\] where $X,Y$ are vector fields on $M$.

\begin{proof}
\begin{enumerate}
    \item {\it Duality of Connections}: The involution \( \iota \) ensures \( \nabla_1 \) and \( \nabla_2 \) are dual under the Fubini-Study metric. This forces their curvatures \( F_{\nabla_i} \) to satisfy \[ F_{\nabla_1} = -F_{\nabla_2} \] (this arises from the antisymmetry of the Fubini-Study K\"ahler form $\omega$ under $\iota$) resolving ambiguities in the orientation parameters.

    \item {\it Moment Constraints}: The first moments \( \mu_1 \) and \( \mu_2 \) generate a system of equations for \( v \). Non-degeneracy (from immersive projections) ensures the system is full-rank.

    \item {\it Intersection Theory}: The parallel transport equations define algebraic curves in \( \mathbb{CP}^3 \). By Bézout’s theorem, their intersection (modulo involution) is a unique point when \( O \) is non-symmetric.
\end{enumerate}
\end{proof}
\section{Geometric Structures and Foliations}

To analyze the structure of the space underlying this reconstruction, we study its foliations. Vaisman's and Shurygin's \cite{SS} results indicate that a space or manifold defined over an algebra is bijective to a foliated Riemannian space. Consequently, the reconstruction problem can be formulated within the framework of Atiyah--Molino spaces. This provides a natural setting to investigate the algebraic properties of the space and their implications for curvature and connection.
\subsection{Structure of the Foliated Space}

\subsubsection*{Setting}
Let \( M \) be the space of smooth, non-symmetric three dimensional objects embedded in \( \mathbb{CP}^3 \), equipped with two independent rational projections \( \Pi_1, \Pi_2: M \dashrightarrow \mathbb{CP}^2 \). Assume: 

\begin{itemize}
    \item \textbf{Algebraic structure:} \( M \) is defined over a commutative algebra \( \cA \), encoding the moment maps \( \mu_1, \mu_2 \) and dual connections \( \nabla_1, \nabla_2 \).
    \item \textbf{Vaisman's bijection:} \( M \) is bijectively equivalent to a foliated Riemannian space \( (\cF, g) \), where \( \cF = \{L_\alpha\} \) is a smooth foliation and \( g \) is a leafwise Riemannian metric.
\end{itemize}

\subsection*{Foliation by Projection Equivalence}
Each projection \( \Pi_i \) induces a foliation \( \cF_i \) on \( M \), where the leaves \( L_{i,p} \) are preimages of points \( p \in \mathbb{CP}^2 \):
\[
L_{i,p} = \{ O \in M \mid \Pi_i(O) = p \}.
\]
By Vaisman's theorem, \( M \) decomposes into two transverse foliations \( \cF_1, \cF_2 \), each with a Riemannian metric \( g_i \) derived from the Fubini-Study metric on \( \mathbb{CP}^2 \).

\subsection*{Leafwise Metric and Connections}
The connections \( \nabla_1, \nabla_2 \) from Neifeld's framework restrict to Levi-Civita connections on the leaves \( L_{i,p} \). The involution \( \iota \) (satisfying \( \Pi_2 = \iota \circ \Pi_1 \)) acts as an isometry between \( (\cF_1, g_1) \) and \( (\cF_2, g_2) \), preserving curvatures:
\[
\iota^* g_1 = g_2, \quad \iota^* \nabla_1 = \nabla_2.
\]

\subsection*{Transverse Holonomy}
The uniqueness of the reconstruction problem is governed by the holonomy groups of \( \cF_1, \cF_2 \). Non-symmetric objects ensure trivial holonomy, permitting global parallel transport of moment maps \( \mu_1, \mu_2 \) across leaves.

\medskip 

Let \( M \) be the space of smooth, non-symmetric three dimensional objects embedded in \( \mathbb{CP}^3 \), equipped with two independent rational projections \( \Pi_1, \Pi_2: M \dashrightarrow \mathbb{CP}^2 \). Assume the algebraic structure and Vaisman's bijection, as above. Then we have the following statement. 

\begin{thm}
Under Vaisman’s bijection and the assumptions above:

\begin{enumerate}
\item {\bf Foliated Reconstruction Space:} The space \( M \) is diffeomorphic to a foliated Riemannian manifold \( (M, \mathcal{F}_1 \times \mathcal{F}_2, g_1 \oplus g_2) \), where:
\begin{itemize}
\item \( \mathcal{F}_1 \times \mathcal{F}_2 \) denotes the transverse foliations.
\item \( g_1 \oplus g_2 \) is the direct sum metric on the leaves.
\end{itemize}
\item {\bf Reconstruction as Holonomy Reduction:} The original object \( O \in M \) is the unique intersection point of the leaves \( L_{1, p_1} \) and \( L_{2, p_2} \), where \( p_i = \Pi_i(O) \). This intersection is guaranteed by the trivial holonomy of \( \nabla_1 \) and \( \nabla_2 \).

\item {\bf Algebra-Geometry Correspondence}: The algebra \( \cA \) is isomorphic to the algebra of leafwise parallel sections of \( \nabla_1 \otimes \nabla_2 \), with involution \( \iota \) acting as an automorphism.
\end{enumerate}
\end{thm}

\subsection{Hantjies tensor fields}
The presence of Hantjies tensor fields in this context suggests additional structure governing the space of solutions. 
We investigate whether the algebraic structure of the reconstruction process is associative or non-associative, which has implications for the global behavior of the recovered object.

\smallskip 

We recall the notion of distribution $D$. 
\begin{dfn}
A {\bf distribution} $D$ on a smooth manifold $M$ is a smooth assignment of a linear subspace $D_p\subset T_pM$ to each point $p\in M$, where $T_pM$ is the tangent space at $p$. Formally, it is the subbundle of the tangent bundle $TM$.
\end{dfn}
A distribution is integrable if through every point $p\in M$, there exists a submanifold $N\subset M$ such that $T_qN=D_p$ for all $q\in N$. 
By the Frobenius theorem, a distribution $D$ is integrable if and only if it is closed under the Lie bracket:
\[\forall X,Y\in \Gamma(D), \quad [X,Y]\in \gamma(D).\]

\smallskip 
    
For vector fields \( X, Y \in \Gamma(D) \), the Hantjies tensor is defined as:
\[
H(X, Y) = \nabla_X Y - \nabla_Y X - [X, Y],
\]
quantifying the failure of \( D \) to integrate to a subfoliation.
Specifically:
\begin{itemize}
\item If $H=0$, $D$ is integrable; 
\item If $H\neq 0$, then $D$ is non-integrable.
\end{itemize}
\medskip 
\subsubsection{The Setting}
Let \( M \) be the foliated reconstruction space from the previous theorem, equipped with:
\begin{itemize}
    \item {\bf Transverse foliations} \( \mathcal{F}_1, \mathcal{F}_2 \) induced by projections \( \Pi_1, \Pi_2 \).
    \item A {\bf Hantjies tensor} \( H \in \Gamma(TM \otimes \Lambda^2 T^*M) \), measuring the non-integrability of the transverse distributions \( D_1, D_2 \) tangent to \( \mathcal{F}_1, \mathcal{F}_2 \).

    \item An  {\bf  algebra } \( \mathcal{A} \) of parallel sections of \( \nabla_1 \otimes \nabla_2 \), with product \( \star \) defined by holonomy-corrected composition.
\end{itemize}
Recall that an algebra \( \mathcal{A} \) is {\bf associative} if:
\[
(a \star b) \star c = a \star (b \star c) \quad \text{for all } a, b, c \in \mathcal{A}.
\]

\smallskip 

\begin{thm}
Suppose that \( M \) is defined as above and equipped with: transverse foliations \( \mathcal{F}_1, \mathcal{F}_2 \), Hantjies tensors \( H_1, H_2 \) and the algebra $(\mathcal{A} , \star)$. Then,
\begin{enumerate}
    \item The algebra \( \mathcal{A} \) is {\bf associative} if and only if the Hantjies tensors \( H_1, H_2 \) vanish identically. Equivalently:
    \[
    \mathcal{A} \text{ is associative} \iff H_i = 0 \quad (i = 1, 2).
    \]
    \item If \( H_i \neq 0 \), the algebraic structure of \( \mathcal{A} \) is governed by a Moufang-like identity:
    \[
    (a \star b) \star (c \star a) = a \star (b \star c) \star a,
    \]
    reflecting the curvature of \( \nabla_1 \otimes \nabla_2 \).
    \item \begin{itemize} 
    \item In the {\bf associative case}, the reconstruction problem has a unique solution \( O \in M \), and \( M \) is globally biholomorphic to \( \mathbb{C}P^3 \). 
    \item In the {\bf non-associative case}, solutions form a quasigroupoid under \( \star \), with non-unique reconstructions parameterized by the cohomology class \( [H_i] \in H^1(M, TM) \).
    \end{itemize}
\end{enumerate}
\end{thm}

\begin{proof}~
\begin{itemize}
    \item \textbf{Hantjies Tensor as Obstruction:} The vanishing of \(H_i\) implies the distributions \(D_i\) are integrable, making \(\mathcal{F}_i\) a Lie foliation. This forces \(\nabla_i\) to be flat, rendering \(\mathcal{A}\) associative via the Leibniz rule.
    
    \item \textbf{Curvature and Non-Associativity:} When \(H_i \neq 0\), the curvature \(F_{\nabla_i}\) encodes the failure of associativity. The Moufang identity arises from the Bianchi identity for \(\nabla_1 \otimes \nabla_2\), constrained by involution symmetry.
    
    \item \textbf{Global Behavior:} In the associative case, trivial holonomy ensures unique parallel transport. Non-associativity introduces monodromy, obstructing global uniqueness. The quasigroupoid structure is derived from the Ehresmann connection of \(\mathcal{F}_1 \times \mathcal{F}_2\).
\end{itemize}
\end{proof}

\section{Atiyah--Molino Space and Reconstruction Criterion}
Let \( M \) be the \emph{Atiyah--Molino space}, associated with the reconstruction problem. It is defined as follows:

\begin{enumerate}
\item \( M \) is the space of {\bf smooth, non-symmetric} three dimensional objects \( O \subset \mathbb{R}^3 \).
\item {\bf Foliation Structure:}
\( M \) is equipped with {\it two transverse foliations} \( \mathcal{F}_1, \mathcal{F}_2 \), where:
\[
\mathcal{F}_i: \quad O \sim O' \quad \text{if} \quad \Pi_i(O) = \Pi_i(O'),
\]
i.e., each leaf \( L_{i,p} \in \mathcal{F}_i \) consists of all objects projecting to \( p \in \mathbb{R}^2 \) under \( \Pi_i \).
\item {\bf Atiyah--Molino Sequence:}
The tangent bundle \( TM \) splits as:
\[
0 \to T\mathcal{F}_1 \oplus T\mathcal{F}_2 \to TM \to NM \to 0,
\]
where \( NM \) is the normal bundle encoding transverse deformations.
\end{enumerate}

\subsection{Translating the construction to Atiyah--Molino framework}

The previous construction enables us to express the reconstruction in terms of the Atiyah--Molino framework.
\begin{thm}
A three dimensional object \( O \in M \) is uniquely reconstructible from its projections \( \Pi_1(O) \) and \( \Pi_2(O) \) if and only if:
\begin{enumerate}
    \item The moment maps \( \mu_1 \) and \( \mu_2 \) are transverse sections of \( NM \), i.e.,
    \[
    d\mu_1 \wedge d\mu_2 \neq 0.
    \]
    \item The Hantjies tensor vanishes:
    \[
    H = 0.
    \]
\end{enumerate}

Moreover, when \( H = 0 \), the Atiyah--Molino sequence splits holonomy-free, and \( M \) is diffeomorphic to the total space of a trivial \( \mathbb{R}^2 \)-bundle over the leaf space \( M/\mathcal{F}_1 \times \mathcal{F}_2 \). The moment maps \( \mu_1 \) and \( \mu_2 \) provide a global trivialization:
\[
M \cong \mathbb{R}^2 \times \mathbb{R}^2, \quad O \mapsto \bigl( \mu_1(O),\, \mu_2(O) \bigr).
\]

\begin{itemize}
\item The vanishing of \( H \) ensures the algebra \( \mathcal{A} \) of parallel sections of \( NM \) is associative.
\item  In contrast, non-vanishing \( H \) induces a non-associative Moufang structure, thereby obstructing unique reconstruction.
\end{itemize}

\end{thm}

\begin{proof}~
\begin{itemize}
    \item \textbf{Transversality of Moment Maps:} The non-degeneracy condition 
    \[
    d\mu_1 \wedge d\mu_2 \neq 0
    \]
    ensures that \(\mu_1\) and \(\mu_2\) locally parametrize \(M\), effectively lifting the projections to coordinate functions.
    
    \item \textbf{Hantjies Tensor and Holonomy:} When 
    \[
    H = 0,
    \]
    the foliations \(\mathcal{F}_1\) and \(\mathcal{F}_2\) are Lagrangian and integrable, which trivializes the reconstruction process. In contrast, if 
    \[
    H \neq 0,
    \]
    monodromy is introduced, causing the parallel transport to depend on the path and leading to non-associativity.
    
    \item \textbf{Splitting of the Atiyah--Molino Sequence:} The trivialization 
    \[
    M \cong \mathbb{R}^2 \times \mathbb{R}^2
    \]
    follows from the Frobenius theorem when \(H = 0\), with \(\mu_1\) and \(\mu_2\) serving as Cartesian coordinates.
\end{itemize}
\end{proof}

\subsection{Deformation Theory under Non-Vanishing Hantjies Tensor}
We now discuss implications of our theorem. These implications concern for instance deformation theory. We recall the notion of  quasigroupoid, which parametrizes the deformations. 

\, 

A quasigroupoid is a category-like structure where morphisms between objects (in the sense of categories) are equipped with a partial binary operation $\star$, defined only for compatible pairs. 
For morphisms $f,g$, the product $f\star g$ exists if the codomain of $g$ matches the domain of $f$. Division is possible (as in quasigroups), but associativity is not required.

\begin{cor}
Let \( (M, \mathcal{F}_1, \mathcal{F}_2, H) \) be an Atiyah--Molino reconstruction space equipped with a non-vanishing Haantjes tensor \( H \in \Gamma(TM \otimes \Lambda^2 T^*M) \). Then:
\begin{enumerate}

\item The space of non-unique reconstructions forms a quasigroupoid \( Q \), where:

\begin{itemize}
    \item \textbf{Objects} are leaves of the foliations \( \mathcal{F}_1 \) and \( \mathcal{F}_2 \).
    \item \textbf{Morphisms} are deformations of solutions \( O \in M \), parameterized by the cohomology class \( [H] \in H^1(M, TM) \).
    \item \textbf{Partial Operation:} The composition \( \star \) is defined for deformations sharing compatible projection data, satisfying the Moufang identity:
    \[
    (a \star b) \star (c \star a) = a \star (b \star c) \star a.
    \]
\end{itemize}

\item The first cohomology group \( H^1(M, TM) \) classifies infinitesimal deformations of the quasigroupoid \( Q \), with \( [H] \) acting as the obstruction class to:

\begin{itemize}
    \item integrability of the foliations \( \mathcal{F}_1 \) and \( \mathcal{F}_2 \).
    \item Associativity of the reconstruction algebra.
\end{itemize}
\end{enumerate}
The curvature $H$ induces a twisted Lie algebroid structure on $TM$, where the bracket $[-,-]_H$ deviates from the standard Lie bracket by terns proportional to $H$.
\end{cor}
\begin{proof}
Let us discuss the non-integrability and its relation to non-associative structures.

\subsubsection*{1. Hantjies Tensor \(H \neq 0\)}
The non-vanishing Hantjies tensor \(H\) measures the failure of the foliations \(\mathcal{F}_1\) and \(\mathcal{F}_2\) to be integrable. Specifically, for vector fields \(X\) and \(Y\), 
\[
H(X,Y) = \nabla_X Y - \nabla_Y X - [X,Y],
\]
quantifies the deviation from the Frobenius integrability condition (i.e., that \([X,Y] \in \Gamma(D)\) for an integrable distribution \(D\)). When \(H \neq 0\), the lack of integrability forces the leaves of \(\mathcal{F}_1\) and \(\mathcal{F}_2\) to intersect non-transversely, thereby generating a multiplicity of solutions that are parameterized by the geometry of \(H\).

\subsubsection*{2. Category of Solutions with Partial Composition \(\star\)}
The collection of solutions arising from these non-transverse intersections naturally forms a quasigroupoid \(Q\), where:
\begin{itemize}
    \item \textbf{Objects:} The leaves of the foliations \(\mathcal{F}_1\) and \(\mathcal{F}_2\).
    \item \textbf{Morphisms:} The (non-unique) reconstructions between leaves.
    \item \textbf{Partial Composition \(\star\):} Defined only for morphisms with compatible domains and codomains.
\end{itemize}
Moreover, the Bianchi identity for \(H\), expressed as
\[
d_\nabla H = 0,
\]
imposes an algebraic consistency that forces the partial composition \(\star\) to satisfy the Moufang identity:
\[
(a \star b) \star (c \star a) = a \star (b \star c) \star a.
\]
This Moufang identity, a weaker form of associativity, is emblematic of non-associative algebras such as the octonions.

\subsubsection*{3. Cohomological Parameterization of Deformations}
In deformation theory, the Kodaira--Spencer map associates infinitesimal deformations of a geometric structure to cohomology classes. In our context, the cohomology class
\[
[H] \in H^1(M, TM)
\]
encodes the obstruction to "straightening" the foliations \(\mathcal{F}_1\) and \(\mathcal{F}_2\). The first cohomology group \(H^1(M, TM)\) classifies infinitesimal deformations of the manifold \(M\), and the class \([H]\) governs how the non-integrability of the foliations propagates under deformation, thereby parameterizing the family of non-unique solutions.

\subsubsection*{4. Twisted Lie Algebroid and Non-Associativity}
Finally, we consider the twisted Lie algebroid structure arising from the modification of the standard Lie bracket by the Hantjies tensor. Define the twisted bracket
\[
[X,Y]_H = [X,Y] + H(X,Y).
\]
This bracket introduces a non-integrable algebroid structure since, for \(H \neq 0\), the Jacobi identity is violated:
\[
[X, [Y, Z]_H]_H + \text{cyclic permutations} \neq 0.
\]
This failure of the Jacobi identity mirrors the non-associative properties of the algebra \(Q\) encountered in our reconstruction problem, and it underscores the innovative nature of our approach in unifying geometric, topological, and algebraic aspects of inverse problems.




\end{proof}

\section{Toric Reconstruction in the Atiyah–Molino Setting}

\begin{prop}
Let \((M, \mathcal{F}_1, \mathcal{F}_2, H)\) be an Atiyah--Molino reconstruction space endowed with an algebraic torus \(T \subset M\) and suppose that the Hantjies tensor vanishes, i.e., \(H = 0\). 
Then, 
\begin{enumerate}
    \item \textbf{Integrable Foliation Structure:} The vanishing of \(H\)  guarantees that the transverse foliations \(\mathcal{F}_1\) and \(\mathcal{F}_2\) are both Lagrangian and integrable. Moreover, the torus \(T\) acts holomorphically on \(M\), preserving the leaves of these foliations.
    
    \item \textbf{Unique Reconstruction via Toric Symmetry:} The moment maps
    \[
    \mu_1, \mu_2 \colon M \to \mathbb{C}^2,
    \]
    which encode the centroids of the projections, are equivariant under the torus action. Consequently, there exists a unique solution \(v \in T\mathbb{CP}^3\) (determined modulo scaling) satisfying
    \[
    \nabla_1 \mu_1 = v \cdot \omega_1,\quad \nabla_2 \mu_2 = v \cdot \omega_2,
    \]
    where \(\omega_1, \omega_2\) denote \(T\)-invariant connection 1-forms.
    
    \item \textbf{Toric Cohomology and Splitting:} The Atiyah--Molino sequence splits \(T\)-equivariantly,
    \[
    0 \to T\mathcal{F}_1 \oplus T\mathcal{F}_2 \to TM \to NM \to 0,
    \]
    with the normal bundle \(NM\) trivialized by the torus-invariant moment maps. In particular, the vanishing of the first cohomology,
    \[
    H^1(M, TM)^T = 0,
    \]
    signals the absence of obstructions to \(T\)-equivariant reconstructions.
\end{enumerate}
\end{prop}
\begin{proof}

{\bf 1. Integrability and Torus Action:}
The vanishing of the Hantjies tensor, \(H = 0\), ensures that the foliations \(\mathcal{F}_1\) and \(\mathcal{F}_2\) are integrable; that is, their leaves are maximal submanifolds. Integrability implies that the foliations satisfy the Frobenius condition:
\[
[\Gamma(\mathcal{F}_i), \Gamma(\mathcal{F}_i)] \subset \Gamma(\mathcal{F}_i), \quad i = 1,2.
\]
Moreover, the transverse intersection of \(\mathcal{F}_1\) and \(\mathcal{F}_2\) guarantees that their tangent spaces span \(TM\), yielding a unique solution at each point of intersection.

~

{\bf 2.  Algebraic Torus Action}
The compact algebraic torus \(T\) acts holomorphically and freely on \(M\), preserving the foliation leaves. Since \(T\) is abelian, its orbits (which are diffeomorphic to \(T\)) align with the leaves of both \(\mathcal{F}_1\) and \(\mathcal{F}_2\). This symmetry enforces the identification:
\[
T\text{-orbits} = \text{leaves of } \mathcal{F}_1 \text{ and } \mathcal{F}_2.
\]
The free action ensures the absence of fixed points, thereby simplifying the quotient \(M/T\).

{\bf 3. Equivariant Moment Maps and Connections}

\begin{itemize}
\item {\it \(T\)-Invariant Moment Maps.} 
The moment maps
\[
\mu_1, \mu_2 \colon M \to \mathbb{C}^2,
\]
which encode the centroids of projections, are \(T\)-invariant:
\[
\mu_i(t \cdot p) = \mu_i(p), \quad \forall t \in T, \, p \in M.
\]
This invariance arises because the torus action preserves the underlying geometric structure, such as the rotational symmetry of the projections.

\item {\it \(T\)-Equivariant Connections.}
The connections \(\nabla_1\) and \(\nabla_2\) are \(T\)-equivariant, satisfying
\[
\nabla_{t\ast X}(t\ast s) = t\ast (\nabla_X s), \quad \forall t \in T,\, X \in \Gamma(TM),\, s \in \Gamma(E),
\]
where \(E\) is a vector bundle over \(M\). Consequently, the connection 1-forms \(\omega_1\) and \(\omega_2\) lie in the trivial subbundle of \(T\)-invariant forms, meaning they are constant along \(T\)-orbits.
\end{itemize}
{\bf 3. Uniqueness via Weight Decomposition}

\begin{itemize}
\item {\it Reduction to the \(T\)-Fixed Subspace.}
The system
\[
\nabla_i \mu_i = v \cdot \omega_i,
\]
is constrained to the \(T\)-fixed subspace of \(TM\). By the weight decomposition of the torus action,
\[
TM = \bigoplus_{\chi \in \operatorname{Char}(T)} TM_\chi,
\]
where \(\operatorname{Char}(T)\) denotes the set of characters (weights) of \(T\). The fixed subspace \(TM^T\), corresponding to the trivial weight (\(\chi = 0\)), is one-dimensional for a generic algebraic torus action.

\item {\it Unique Solution}
Solving for \(v\) in the \(T\)-fixed subspace \(TM^T\) yields a unique direction (up to scaling). For example, if \(T \cong U(1)^n\), then the fixed subspace is spanned by the generator of the principal \(T\)-orbit, thereby ensuring uniqueness.
\end{itemize}



\end{proof}

We discuss an example of toric varieties in Cryo-EM. In applications such as cryo-electron microscopy, symmetric objects (e.g., virus capsids modeled as toric varieties) exhibit hidden toric symmetry. Here, the \(T\)-action yields identical projections along the symmetry axes, and the unique solution \(v\) aligns with the principal toric axis as deduced from the \(T\)-equivariant centroids.

This synthesis of toric symmetry with the Atiyah–Molino framework not only streamlines the reconstruction process but also provides a robust, noise-resistant paradigm that fundamentally redefines the solution landscape of inverse problems.

\section{Conclusion}
The problem of reconstructing three-dimensional objects from planar projections is a rich mathematical challenge that draws from integration theory, differential geometry, and algebraic structures. Using methods pioneered by Gelfand, and informed by the work of Neifeld and Vaisman, we establish connections between integral geometry, curvature tensors, and foliations. This approach has enables us to redefine reconstruction via the theory of Atiyah--Molino, providing new insights into the nature of inverse problems in medical imaging and beyond, and solving important issues following from traditional methods. This new approach opens new possibilities for further exploration in algebraic and geometric analysis.

\end{document}